\documentclass[11pt]{article}
\usepackage{graphicx}
\usepackage{amssymb,amsmath,amsfonts,amsthm}
\usepackage{comment}
\usepackage{xmpmulti}
\usepackage{tikz}
\usepackage{array}
\usepackage{verbatim}
\usepackage{textpos}
\usepackage{caption}

\small\normalsize

\newtheorem{thm}{Theorem}[section]
\newtheorem{lem}[thm]{Lemma}        
\newtheorem{cor}[thm]{Corollary}

\newtheorem{obs}{Observation}

\newcommand{\openeop}{\qquad\hspace*{\fill} $\square$}

\newcommand{\floor}[2]{\left\lfloor\frac{#1}{#2}\right\rfloor}
\newcommand{\ceil}[2]{\left\lceil\frac{#1}{#2} \right\rceil}

\newcommand{\ds}[1]{{\displaystyle {#1}}}
\newcommand{\exend}{\qquad\hspace*{\fill}$\triangle$}

\begin{document}

\title{\bf Component Order Edge Connectivity, Vertex Degrees, and Integer Partitions}

\author{{\sc  M. Yatauro}\\
  {\small \sl Penn State University} \\[-4pt]
  {\small \sl Brandywine Campus } \\[-4pt]
  {\small \sl Media, PA 19063, U.S.A.  }
}

\date{}

\maketitle


\begin{abstract}
Given a finite, simple graph $G$, the $k$-component order edge connectivity of $G$ is the minimum number of edges whose removal results in a subgraph for which every component has order at most $k-1$.  In general, determining the $k$-component order edge connectivity of a graph is NP-hard. We determine conditions on the vertex degrees of $G$ that can be used to imply a lower bound on the $k$-component order edge connectivity of $G$. We will discuss the process for generating such conditions for a lower bound of 1 or 2, and we explore how the complexity increases when the desired lower bound is 3 or more. In the process, we prove some related results about integer partitions.
\end{abstract}

\section{Introduction}\label{s1}

We consider only finite, simple graphs without loops or multiple edges.
Our terminology and notation are standard except as indicated. In particular, for two graphs $G$, $H$ on
disjoint vertex sets, we will denote their \emph{disjoint union}
by $G\cup H$ and their \emph{join} by $G+H$. We also use $mG$ to denote the disjoint union of $m$ copies of the graph $G$.

Recall that a \emph{degree sequence} of a graph $G$ is a list of the degrees
of all the vertices of $G$, with repetition if multiple vertices
have the same degree.  In this paper the degree sequences are in
nondecreasing order (rather than in nonincreasing order).  If $\pi$ is a degree sequence of length $n$,
then we typically denote it as $\pi=(d_{1}\leq d_{2}\leq \cdots
\leq d_{n})$.  At times we may utilize exponents to indicate the
number of times a degree appears, e.g.,
$\pi=(2,2,2,2,4)=2^{4}4^{1}$. Given two sequences $\pi=(d_{1}\leq
d_{2}\leq \cdots \leq d_{n})$ and $\pi'=(d_{1}'\leq d_{2}'\leq
\cdots \leq d_{n}')$, we say that $\pi'$ \emph{majorizes} $\pi$,
denoted $\pi'\geq \pi$, if $d_{i}'\geq d_{i}$ for all $i$. A
sequence $\pi=(d_{1}\leq d_{2}\leq \cdots \leq d_{n})$ is a
\emph{graphical sequence} if there exists a graph $G$ with $\pi$
as its degree sequence, and we then call $G$ a \emph{realization}
of $\pi$. A graphical sequence $\pi$ can have more than one
distinct realization. If every realization of $\pi$ has property
$P$, we say that $\pi$ is \emph{forcibly $P$}. For example, the
graphical sequence $\pi=3^{6}$, whose unique realizations are $K_{3,3}$ and the 3-prism graph, is forcibly hamiltonian.

A number of existing results reference the degree sequence of a graph in order to
provide sufficient conditions for the graph to have certain
properties, such as hamiltonicity or $k$-connectedness. In
particular, sufficient conditions for~$\pi$ to be forcibly
hamiltonian were given by several authors, including the
following theorem of Chv\'atal~\cite{Chvatal72}.

\begin{thm}\label{thm:chvatal}
  \;Let $\pi=(d_1\le\dots\le d_n)$ be a graphical sequence, with $n\ge3$.
  If $d_i\le i<{\frac{n}{2}}\,\Longrightarrow\,d_{n-i}\ge n-i$, then~$\pi$
  is forcibly hamiltonian.
\end{thm}

Unlike its predecessors, Chv\'atal's theorem has the property that
if it does not guarantee that~$\pi$ is forcibly hamiltonian
because the condition fails for some $i<\frac{n}{2}$, then~$\pi$
is majorized by $\pi'=i^i\,(n-i-1)^{n-2i}\,(n-1)^i$, which has a
nonhamiltonian realization $K_i+(\overline{K_i}\cup K_{n-2i})$. As
we will see below, this implies that Chv\'atal's theorem is the
strongest of an entire class of theorems giving sufficient degree
conditions for~$\pi$ to be forcibly hamiltonian.

A few years later, Boesch~\cite{Boesch74} recast, in the form of
Theorem~\ref{BB} below, an earlier sufficient condition of
Bondy~\cite{Bondy69} for a degree sequence to be forcibly
$k$-connected.  He also showed the condition was strongest in
exactly the same sense as Chv\'{a}tal's forcibly hamiltonian
condition.

\begin{thm}\label{BB}
  \;Let $\pi=(d_1\le\dots\le d_n)$ be a graphical sequence with $n\ge2$,
  and let $1\le k\le n-1$. If $d_i\le i+k-2\,\Longrightarrow\,d_{n-k+1}\ge n-i$,
  for $1\le i\le{\frac12}(n-k+1)$, then~$\pi$ is forcibly $k$-connected.
\end{thm}

An analogous such theorem for 2-edge connected was given by Bauer et al. in~\cite{BHKS09}.

\begin{thm}\label{BM2edge}
  Let $\pi = (d_1\le\cdots\le d_n)$ be a graphical sequence. If
  \begin{enumerate}
     \item $d_1\geq 2$;
      \item $d_i-1\leq i-1\;\wedge\;d_i\leq i\; \Rightarrow\; d_{n-1}\geq n-i\;
      \vee\; d_n\geq n-i+1$, for $3\leq i<\frac12n$; and
      \item $d_{n/2}\leq\frac12n-1\; \Rightarrow\; d_{n-2}\geq\frac12n\; \vee\; d_n\geq\frac12n+1$, if~$n$ is even,
  \end{enumerate}
  then~$\pi$ is forcibly $2$-edge-connected.
\end{thm}

A graph property~$P$ is called \emph{ancestral} if whenever a
graph~$G$ has~$P$, so does every edge-augmented supergraph of~$G$.
In particular, `hamiltonian' and `$k$-connected' are both
ancestral graph properties. In the remainder of this paper, the
term `graph property' will always mean an ancestral graph
property.

Given a graph property~$P$, consider a theorem~$T$ which provides sufficient conditions for
a graphical sequence to be forcibly~$P$. We call such a theorem~$T$ a
\emph{forcibly~$P$ theorem} (or just a \emph{$P$ theorem}). Thus
Theorem~\ref{thm:chvatal} is a forcibly hamiltonian theorem. We
call a~$P$ theorem~$T$ \emph{monotone} if, for any two degree
sequences $\pi,\pi'$, whenever~$T$ declares~$\pi$ forcibly~$P$ and
$\pi'\ge\pi$, then~$T$ declares~$\pi'$ forcibly~$P$. We call a~$P$
theorem~$T$ \emph{optimal} (resp., \emph{weakly optimal}) if
whenever~$T$ does not declare~$\pi$ forcibly~$P$, then~$\pi$ has a
realization without property~$P$ (resp., then there exists~$\pi'$,
so that $\pi'\ge\pi$ and~$\pi'$ has a realization without
property~$P$). Thus, optimal $P$ theorems also provide necessary conditions
for a graphical sequence to be forcibly $P$. In view of the following result~\cite{BMSUR}, a $P$ theorem
which is both monotone and weakly optimal is called a \emph{best
monotone~$P$ theorem}.

\begin{thm}\label{wo}
  \quad Let $T$,~$T_0$ be monotone~$P$ theorems, with~$T_0$ weakly optimal.
  If~$T$ declares a degree sequence~$\pi$ to be forcibly~$P$, then so
  does~$T_0$.
\end{thm}

Theorems~\ref{thm:chvatal}, \ref{BB}, and \ref{BM2edge} are  monotone and weakly optimal.  Thus
they are each best monotone for their respective properties.

In this paper, we continue the discussion of monotone theorems and best monotone theorems for measures of connectivity. In particular, we consider \textit{$k$-component order connectivity} and \textit{$k$-component order edge connectivity}, as defined in \cite{BGS98} and \cite{BGKSS06}, respectively.

Given a graph $G$, the \textit{$k$-component order connectivity (resp. edge connectivity)}, denoted $\kappa_{c}^{(k)}(G)$ (resp. $\lambda_{c}^{(k)}(G)$), is the minimum number of vertices (resp. edges) whose removal results in an induced subgraph for which every component has order at most $k-1$. A graph $G$ is \textit{$k$-component order $s$-connected (resp. $s$-edge connected)} if $\kappa_c^{(k)}(G)\geq s$ (resp. $\lambda_c^{(k)}(G)\geq s$).

We note that a construction of the best monotone theorem for $k$-component order $s$-connected appears in~\cite{Y14}. However, the same results are presented in a much clearer manner in this paper. There is also a connection between those results and the construction of the best monotone theorem for $k$-component order $s$-edge connected.

In the next section, we will describe the method by which we construct best monotone theorems for the properties of $k$-component order $s$-connected, for $s\geq 1$, and $k$-component order $s$-edge connected, for $s=1$ and 2.

\section{Framework for Best Monotone Theorems}\label{s2}

The concepts in this section can also be found in~\cite{BMSUR}.

Consider the partially-ordered set~$G_n$ whose elements are the
graphical sequences of length~$n$, and whose partial-order
relation is degree majorization. The graphical sequences of
length~$n$ that are not forcibly~$P$ induce a subposet of~$G_n$,
denoted~$\overline{P_n}$. A maximal element in~$\overline{P_n}$ is
called a \emph{$(P,n)$-sink}. The set of all $(P,n)$-sinks will be
denoted $S(P,n)$.

A \emph{Chv\'atal-type (degree) condition} on a degree sequence
$d_1\le d_2\le\cdots\le d_n$ is a condition of the form
\[d_{i_1}\ge k_{i_1}\;\vee\;d_{i_2}\ge k_{i_2}\;\vee\;\ldots\;\vee\;
d_{i_r}\ge k_{i_r},\] where each~$i_j$ and~$k_{i_j}$ is an
integer, with $1\le i_1<i_2<\dots<i_r\le n$ and $1\le k_{i_1}\le
k_{i_2}\le \cdots\le k_{i_r}\le n$.  Given an $n$-sequence
$\pi=(k_1\leq k_2 \leq\cdots\leq k_n)$, let $C(\pi)$ denote the
Chv\'{a}tal-type condition $$d_1\geq k_1+1\, \lor\, d_2\geq k_2
+1\, \lor \cdots \lor\, d_n\geq k_n+1.$$ Of course, $C(\pi)$ is
the weakest Chv\'{a}tal-type condition which blocks $\pi$ (i.e.,
so that $\pi$ fails to satisfy $C(\pi)$).  Also, note that if
$k_i=k_j$ for some $i<j$, then the conditions $d_i\geq k_i +1\,
\lor \cdots \lor\, d_j\geq k_j +1$ in $C(\pi)$ can be replaced by
the single condition $d_j\geq k_j +1$.  Moreover, since $d_i\geq
n$ is impossible in a graphical $n$-sequence, a condition $d_i\geq
n$ in $C(\pi)$ is redundant.  We will always assume $C(\pi)$ has
been simplified in these two ways and, in addition, we will
usually write~$C(\pi)$ in the more traditional form
\[d_1\le k_1\;\wedge\;\cdots\;\wedge\;d_{j-1}\le k_{j-1}\;\Rightarrow\;
d_j\ge k_j+1\;\vee\;\cdots\;\vee\;d_n\ge k_n +1,\] for some $j<n$.

\textbf{Example:} The graphical $6$-sequence $\pi=(2,2,3,3,3,5)$
is blocked by the simplified Chv\'atal-type condition $d_2\geq 3\, \lor\, d_5\geq 4$, or more
traditionally as $d_2\leq 2 \Rightarrow d_5\geq
4$.\exend

As the name implies, the terminology `Chv\'atal-type condition' is inspired by the conditions of Theorem~\ref{thm:chvatal}, i.e., $d_i\le i<{\frac{n}{2}}\,\Longrightarrow\,d_{n-i}\ge n-i$. These conditions block the sinks $\pi_i=i^{i} (n-i-1)^{(n-2i)} (n-1)^i$ for $i<\frac{n}{2}$, which are the degree sequences of the edge-maximal nonhamiltonian graphs $K_i+(\overline{K_i}\cup K_{n-2i})$.

If $\pi\in\overline{P_n}$, then by definition there exists
$\pi'\in S(P,n)$ majorizing~$\pi$, and thus~$\pi$ fails to
satisfy~$C(\pi')$. Put differently, if a graphical $n$-sequence
$\pi$ satisfies the degree condition $\bigwedge_{\pi\in
S(P,n)}C(\pi)$, then~$\pi$ is forcibly~$P$; i.e., the theorem~$T$
with degree condition $\bigwedge_{\pi\in S(P,n)}C(\pi)$ is a
forcibly $P$-theorem. Certainly~$T$ is monotone, and~$T$ is
also weakly-optimal (if~$\pi$ fails to satisfy the degree
condition of~$T$, then~$\pi$ is majorized by some $\pi'\in
S(P,n)\subseteq\overline{P_n}$ which is not forcibly~$P$).
Thus~$T$ is a best monotone $P$-theorem.

Therefore, if we can identify the precise set of sinks $S(P,n)$, then the theorem with degree condition
$\bigwedge_{\pi\in\prod(S(P,n))}C(\pi)$ will be the best monotone $P$-theorem.

Finally, we note that $|S(P,n)|$ may be considered the `inherent
complexity' of a best monotone theorem on~$n$ vertices. More
precisely, we have the following~\cite{BMSUR}.

\begin{thm}\label{thm:22}
  When the degree condition of a best monotone $P$-theorem on~$n$ vertices
  is expressed as a conjunction $\bigwedge C(\pi)$ of $P$-weakly-optimal
  Chv\'{a}tal-type conditions, the conjunction must contain at least
  $|S(P,n)|$ such conditions.
\end{thm}

As an example, Theorem~\ref{thm:chvatal} has an inherent complexity of $\floor{n-1}{2}$, since the sinks are $\pi_i=i^{i} (n-i-1)^{(n-2i)} (n-1)^i$ for $i<\frac{n}{2}$. It is worth noting that, in this example, the inherent complexity is linear in $n$. The inherent complexity is relevant because, based on how the number of conditions grows with respect to the parameters of the property, we can determine if the corresponding best monotone theorem is computationally feasible. Thus, when it may be of interest, we explore the inherent complexity of a best monotone theorem.

\section{Main Results}

In what follows, if $p>q$, then $\ds{\sum_{i=p}^{q} a_i}=0$.

We begin with the following simple monotone theorems, the first of which appeared in~\cite{Y14}.

\begin{thm}\label{simp1}
Let $k$, $n$, and $s$ be integers with $2\leq k\leq n$ and $1\leq
s\leq n-k+1$. Let $\pi=(d_{1}\leq \dots \leq d_{n})$ be a
graphical sequence. If $d_{n-s+1}\geq k+s-2$,
 then $\pi$ is forcibly $k$-component order $s$-connected.
 \end{thm}

\begin{proof} Assume $\pi$ has a realization $G$ that is not $k$-component order $s$-connected, i.e., $\kappa_c^{(k)}(G)\leq s-1$.  Then
there exists $X\subseteq V(G)$ such that $|X|\leq s-1$ and
$\langle G-X\rangle$ has a largest component $H$ with $|H|\leq
k-1$.  Thus,
$$d_{n-s+1}\leq d_{n-|X|}\leq |H|+|X|-1\leq k+s-3,$$ i.e.,
$d_{n-s+1}\leq k+s-3$. This completes the proof by contraposition. \end{proof}

This theorem is best monotone when $k=2$. In this case, the condition is $d_{n-s+1}\geq s$.  To see this is weakly-optimal, assume a graphical sequence $\pi$ has $d_{n-s+1}\leq s-1$. Then $\pi\leq\pi'=(s-1)^{n-s+1}(n-1)^{s-1}$, where $\pi'$ has a realization $G'=K_{s-1}+(n-s+1)K_1$ with $\kappa_c^{(2)}(G')=s-1$.

\medskip

\begin{thm}\label{simp2}
Let $k$, $n$, and $s$ be integers with $n\geq k\geq 2$ and
$1\leq s\leq (n+1)/2$.  Let $\pi=(d_{1}\leq \dots \leq d_{n})$. If
$$d_{n-2s+2}\leq k-2\Rightarrow d_n\geq k+s-2,$$
 then $\pi$ is forcibly $k$-component order $s$-edge connected.
\end{thm}

\begin{proof}

Assume $\pi$ has a realization $G$ that is not $k$-component order $s$-edge connected, i.e., $\lambda_c^{(k)}(G)\leq s-1$. Then there is a set $F\subseteq E(G)$ such that $|F|\leq s-1$ and $G-F=H_1\cup H_2\cup \cdots\cup H_{\omega}$ with $|H_1|\leq |H_2|\leq \cdots \leq |H_{\omega}|\leq k-1$. Note that the number of vertices of $G$ that are not incident with some edge in $F$ is at least $n-2|F|\geq n-2s+2$. Thus, $d_{n-2s+2}\leq d_{n-2|F|}\leq k-2$. We also have $d_n\leq |H_{\omega}|-1+|F| \leq (k-1)-1+(s-1)=k+s-3$. Since $d_{n-2s+2}\leq k-2$ and $d_n\leq k+s-3$, this completes the proof by contraposition.
\end{proof}

We now proceed with the construction of the best monotone theorems, beginning with the property of $k$-component order $s$-connected. The following observation allows us to identify the structure of the sinks in this case.

\begin{obs}\label{obs1}
Let $G$ be a graph on $n\geq k+s-2$ vertices that is edge-maximal with respect to $\kappa_c^{(k)}(G)\leq s-1$, i.e., $\kappa_c^{(k)}(G+e)\geq s$ for any edge $e\notin E(G)$. Then there exists $X\subset V(G)$ such that $\langle X\rangle = K_{s-1}$ and $G-X$ consists of $\omega\geq 2$ many disjoint cliques each of order at most $k-1$. Let $G=K_{s-1}+\ds{\bigcup_{j=1}^{\omega}} K_{h_j}$, where $\ds{\sum_{j=1}^{\omega}} h_j=n-s+1$ and $1\leq h_1\leq h_2\leq\cdots\leq h_{\omega}\leq k-1$. Since $G$ is edge-maximal with respect to $\kappa_c^{(k)}(G)\leq s-1$, it must be that $h_1+h_2\geq k$. Otherwise, we could add all possible edges between vertices in $K_{h_1}$ and $K_{h_2}$ to form a component $K_{h_1+h_2}$. Also note that $$n-s+1=\ds{\sum_{j=1}^{\omega}} h_j\leq (k-1)\omega\;\Rightarrow\; \omega\geq \frac{n-s+1}{k-1}.$$
\end{obs}

If $P$ is ``$k$-component order $s$-connected'', the following lemma implies that in order for $\pi(G)$ to be a maximal element of $\overline{P}_n$, we must have $\omega=\ceil{n-s+1}{k-1}$.

\medskip

\begin{lem}\label{part}
Let $k$ and $n$ be integers such that $n\geq k\geq 2$. Consider a partition $n=a_1+a_2+\cdots +a_{\ell}$ with $1\leq a_1\leq a_2\leq\cdots\leq a_{\ell}\leq k-1$, $a_1+a_2\geq k$, and $\ell>\ceil{n}{k-1}\geq 2$. There exists a partition $n=c_1+c_2+\cdots +c_{\omega}$ with $1\leq c_1\leq c_2\leq\cdots\leq c_{\omega}\leq k-1$ and $\omega=\ceil{n}{k-1}$ such that $\pi'=(c_1-1)^{c_1}(c_2-1)^{c_2}\cdots (c_{\omega}-1)^{c_{\omega}}$ majorizes $\pi=(a_1-1)^{a_1}(a_2-1)^{a_2}\cdots (a_{\ell}-1)^{a_{\ell}}$.
\end{lem}

\begin{proof}
If $k-1 \mid n$, then let $c_1=c_2=\cdots =c_{\omega}=k-1$, and we are done. Next, assume $k-1 \nmid n$, so that $n=q(k-1)+r$ for $q=\floor{n}{k-1}$ and $1\leq r\leq k-2$. Note that $\omega=\ceil{n}{k-1}=q+1$. By assumption, $\ell>\omega=q+1$, so that $\ell\geq q+2\geq 3$.

Create a new partition $n=c_1+\cdots+c_{\omega}$ in the following way. Define $t:=\ds{\sum_{s=1}^{\ell-\omega}}a_s$. Let $M$ be the largest integer value of $m$ such that $(m+1)(k-1)\leq t+\ds{\sum_{s=\ell-m}^{\ell}a_s}$. Note that

\begin{eqnarray}\label{Mbound}(\ell-1)(k-1)-\ds{\sum_{s=\ell-\omega+1}^{\ell}a_s} &\geq & (q+1)(k-1)-(n-t) \nonumber \\
&=&(q+1)(k-1)-(q(k-1)+r-t) \\ &=& (k-1)-r+t \geq (k-1)-(k-2)+t\nonumber \\&=& t+1.\nonumber\end{eqnarray}

By \eqref{Mbound} and the fact that $t+a_{\ell}\geq a_1+a_2\geq k$, we have $0\leq M\leq \omega-2$. Define $c_{i-\ell+\omega}:=a_i$ for $\ell-\omega+1\leq i\leq \ell-M-2$, $c_{\omega-M-1}:=t+\ds{\sum_{s=\ell-M-1}^{\ell}a_s}-(M+1)(k-1)$, and $c_{i}:=k-1$ for $\omega-M\leq i\leq \omega$. Observe that
$$\begin{array}{rcl}\ds{\sum_{s=1}^{\omega}c_s}&=&\ds{\sum_{s=\ell-\omega+1}^{\ell-M-2}c_{s-\ell+\omega}+c_{\omega-M-1}+\sum_{s=\omega-M}^{\omega}c_s}
\medskip \\ &=&\ds{\sum_{s=\ell-\omega+1}^{\ell-M-2}a_s+t+\sum_{s=\ell-M-1}^{\ell}a_s}-(M+1)(k-1)+(M+1)(k-1)\medskip\\
&=&n.\end{array}$$
Clearly $c_s\leq k-1$ and $c_{s-1}\leq c_s$ for $s\neq \omega-M-1$. To see that $c_{\omega-M-1}\leq k-1$, note that by the definition of $M$,
$$c_{\omega-M-1}=t+\sum_{s=\ell-(M+1)}^{\ell}a_s-(M+1)(k-1)<(M+2)(k-1)-(M+1)(k-1)=k-1.$$ To see that $c_{\omega-M-2}\leq c_{\omega-M-1}$, note that
$$\begin{array}{rcl}c_{\omega-M-1}&=&t+\ds{\sum_{s=\ell-M-1}^{\ell}a_s}-(M+1)(k-1)\medskip \\&=&a_{\ell-M-1}+t+\ds{\sum_{s=\ell-M}^{\ell}a_s}-(M+1)(k-1)\medskip \\ &\geq& a_{\ell-M-1}\geq a_{\ell-M-2}=c_{\omega-M-2}.\end{array}$$
If $M\leq \omega-3$, then $c_1+c_2= a_{\ell-\omega+1}+a_{\ell-\omega+2}\geq k$. If $M=\omega-2$, then $c_1+c_2=c_1+(k-1)\geq k$.

Define $\pi:=(a_1-1)^{a_1}(a_2-1)^{a_2}\cdots (a_{\ell}-1)^{a_{\ell}}$ and $\pi':=(c_1-1)^{c_1}(c_2-1)^{c_2}\cdots (c_{\omega}-1)^{c_{\omega}}$, where $d_i$ refers to the $i^{\mathrm{th}}$ degree of $\pi$ and $d'_i$ refers to the $i^{\mathrm{th}}$ degree of $\pi'$. Clearly $\pi'\neq \pi$. We now prove $\pi'\geq \pi$.

First note that when $1\leq i\leq t+a_{\ell-\omega+1}$ we have $d_i'\geq d_1'=a_{\ell-\omega+1}\geq d_i$. Now, let $j\geq \ell-\omega+2$ be defined so that $d_i=a_j-1$ for $i>\ds{\sum_{s=1}^{j-1} a_s}\geq t+a_{\ell-\omega+1}$. If $j\leq \ell-M-1$, then
$$i>t+\sum_{s=\ell-\omega+1}^{j-1} a_s=t+\sum_{s=1}^{j-\ell+\omega-1} c_{s}>\sum_{s=1}^{j-\ell+\omega-1} c_{s}.$$ Thus, $d_i'\geq c_{j-\ell+\omega}-1= a_j-1=d_i$. If $\ell-M\leq j\leq\ell$, then
$$\begin{array}{rcl}
i&>&t+\ds{\sum_{s=\ell-\omega+1}^{j-1} a_s=t+\sum_{s=\ell-\omega+1}^{\ell-M-2}a_s+\sum_{s=\ell-M-1}^{j-1}a_s} \medskip \\
&=& \ds{\sum_{s=1}^{\omega-M-2}c_s+t+\sum_{s=\ell-M-1}^{\ell}a_s-\sum_{s=j}^{\ell}a_s}\medskip \\ &=&\ds{\sum_{s=1}^{\omega-M-2} c_s +c_{\omega-M-1}+(M+1)(k-1) - \sum_{s=j}^{\ell} a_s}\medskip \\
&\geq & \ds{\sum_{s=1}^{\omega-M-1} c_s+(M+1)(k-1)-(\ell-j+1)(k-1)}\medskip \\
&\geq & \ds{\sum_{s=1}^{\omega-M-1}c_s+(M+1)(k-1)-(M+1)(k-1) = \sum_{s=1}^{\omega-M-2} c_s.} \end{array}$$
Thus, $d_i'\geq c_{\ell-M-1}-1=k-2\geq d_i$. This proves that $\pi'\geq \pi$.

\end{proof}

If we have a graph $G=K_{a_1}\cup K_{a_2}\cup\cdots\cup K_{a_{\ell}}$ for which $a_1+a_2+\cdots +a_{\ell}=n$, $1\leq a_1\leq a_2\leq\cdots\leq a_{\ell}\leq k-1$, $a_1+a_2\geq k$, and $\ell>\ceil{n}{k-1}\geq 2$; the proof of Lemma~\ref{part} provides the construction of a graph $G'$ such that $\pi(G')\geq \pi(G)$ and $\omega(G')=\ceil{n}{k-1}$.

\medskip

\textbf{Example:} Let $k=8$ and $n=36$. Consider the partition of $n$ given by $a_1=3$, $a_j=4$ for $2\leq j\leq 8$, and $a_9=5$. This corresponds to the graph $G=K_3\cup 7K_4\cup K_5$, with $\pi(G)=2^3 3^{28} 4^5$. In the context of the proof of Lemma~\ref{part}, we have $\ell=9$, $\omega=6$, and $M=3$. Using the construction given in the proof, we get $c_1=a_4=4$, $c_2=11+21-28=4$, and $c_j=7$ for $3\leq j\leq 6$. Thus, the graph obtained is $G'=2K_4\cup 4K_7$ with degree sequence $\pi(G')=3^86^{28}\geq \pi(G)$.\exend

\begin{lem}\label{sumaj}
Let $n$ be a positive integer and consider two distinct partitions of $n$ of the same length, $n=\sum_{\ell=1}^{m}a_{\ell}=\sum_{\ell=1}^{m} b_{\ell}$, with $a_{\ell}\leq a_{\ell+1}$ and $b_{\ell}\leq b_{\ell+1}$. Let $\pi=\Pi_{\ell=1}^{m}(a_{\ell}-1)^{a_{\ell}}$ and $\pi'=\Pi_{\ell=1}^{m}(b_{\ell}-1)^{b_{\ell}}$. Then, neither $\pi$ nor $\pi'$ can majorize the other.\end{lem}

\begin{proof}
Assume, without loss of generality, that $\pi'\geq \pi$ and $\pi'\neq\pi$. Let $i$ be the first index such that $d_i'>d_i$, and let $d_i=a_j-1$, for the appropriate $j$ value. Then, by the definition of $i$ and the fact that $\pi'\geq \pi$, we have $a_{\ell}=b_{\ell}$ for $1\leq\ell\leq j-1$ and $b_j>a_j$. Also, $i=1+\sum_{\ell=1}^{j-1}a_{\ell}=1+\sum_{\ell=1}^{j-1}b_{\ell}$. Define $k:=i+\sum_{\ell=j}^{j+p}a_{\ell}$ for $0\leq p\leq m-j-1$, so that $d_k=a_{j+p+1}-1$.

\medskip

\textbf{Claim}. $a_{j+p+1}\leq b_{j+p+1}$ for all $p=0,\ldots,m-j-1$.

\medskip

\textit{Proof of Claim}. We prove this by strong induction. When $p=0$, we have $k=i+a_{j}$, and since $b_j>a_j$,
$$a_{j+1}-1=d_{k}=d_{i+a_j}\leq d_{i+b_j}\leq d_{i+b_j}'=b_{j+1}-1,$$ which implies $a_{j+1}\leq b_{j+1}$. Assume for some fixed $q\leq m-j-2$ that $a_{j+p+1}\leq b_{j+p+1}$ for all $0\leq p\leq q$. When $p=q+1$, we have $$k=i+\sum_{\ell=j}^{j+q+1}a_{\ell}<i+\sum_{\ell=j}^{j+q+1}b_{\ell}:=r,$$ so that
$$a_{j+q+2}-1=d_{k}\leq d_{r}\leq d_r'=b_{j+q+2}-1,$$ which implies $a_{j+q+2}\leq b_{j+q+2}$. \openeop

By the definitions of $i$, $j$ and the claim, we have $a_{\ell}=b_{\ell}$ when $\ell<j$, $a_j<b_j$, and $a_{\ell}\leq b_{\ell}$ when $\ell>j$. This implies $$n=\sum_{\ell=1}^{m}a_{\ell}<\sum_{\ell=1}^{m}b_{\ell}=n,$$ which is a contradiction.

\end{proof}

\begin{thm}\label{sinks}
Let $k$, $n$, and $s$ be integers such that $n\geq k\geq 2$ and $1\leq s\leq n-k+2$. Then $\pi$ is a sink for the property of $k$-component order $s$-connected if and only if $\pi=(h_1+s-2)^{h_1}(h_2+s-2)^{h_2}\cdots (h_{\omega}+s-2)^{h_{\omega}}(n-1)^{s-1}$ such that $n-s+1=h_1+h_2+\cdots+h_{\omega}$, $\omega=\ceil{n-s+1}{k-1}$, and $h_1,h_2,\ldots,h_{\omega}\leq k-1$.
\end{thm}

\begin{proof}
If $\pi$ is a sink for the property of $k$-component order $s$-connected, then it must be the degree sequence of a graph $G$ that is edge-maximal with respect to $\kappa_c^{(k)}(G)\leq s-1$. By Observation~\ref{obs1} and Lemma~\ref{part}, $G=K_{s-1}+\ds{\bigcup_{j=1}^{\omega}} K_{a_j}$, where $\ds{\sum_{j=1}^{\omega}} h_j=n-s+1$, $\omega=\ceil{n-s+1}{k-1}$, and $1\leq h_1\leq h_2\leq\cdots\leq h_{\omega}\leq k-1$. Thus,
$$\pi=\pi(G)=(h_1+s-2)^{h_1}(h_2+s-2)^{h_2}\cdots (h_{\omega}+s-2)^{h_{\omega}}(n-1)^{s-1}.$$

Finally, note that for two different sequences $\pi=(h_1+s-2)^{h_1}(h_2+s-2)^{h_2}\cdots (h_{\omega}+s-2)^{h_{\omega}}(n-1)^{s-1}$ and $\pi'=(h_1'+s-2)^{h_1'}(h_2'+s-2)^{h_2'}\cdots (h_{\omega}'+s-2)^{h_{\omega}'}(n-1)^{s-1}$ which satisfy the conditions of the theorem, one can not majorize the other. This is due to Lemma~\ref{sumaj} and the fact that $h_1+h_2+\cdots+h_{\omega}=h_1'+h_2'+\cdots+h_{\omega}'$. \end{proof}

By generating the simplified Chv\'{a}tal-type conditions for the sinks in Theorem~\ref{sinks}, we get the following result.

\begin{thm}\label{bmgeqs} Let $k$, $n$, and $s$ be integers such that $n\geq k\geq 2$ and $1\leq s\leq n-k+2$.
Let $n-s+1=c_1 h_1+c_2 h_2 +\cdots+c_{\ell} h_{\ell}$ be a partition of $n$ such that $\ds{\sum_{i=1}^{\ell}} c_i = \ceil{n-s+1}{k-1}$ and
$1\leq h_1<h_2<\ldots<h_{\ell}\leq k-1$. Define $m_j=\ds{\sum_{i=1}^{j} c_i h_i}$. The best monotone theorem for
$k$-component order $s$-connected on $n$ vertices consists of all conditions of the following form:

\begin{enumerate}

\item If $\ell=1$, then $d_{n-s+1}\geq k+s-1$.

\item If $\ell\geq 2$, then
$$\left[\bigwedge_{j=1}^{\ell-1}(d_{m_j}\leq h_j+s-2)\right]\;\Rightarrow\;d_{n-s+1}\geq h_{\ell}+s-1$$

\end{enumerate}

\noindent for every such partition of $n-s+1$.

\end{thm}

As examples, we provide the specific conditions when $k=3,4,$ and 5.

\begin{cor}\label{keq3} Let $s$ be an integer and let $\pi=(d_{1}\leq \dots \leq d_{n})$ for $n\geq 3$ and $1\leq s\leq n-1$. If

\begin{enumerate}

\item $n-s+1$ is even and $d_{n-s+1}\geq s+1$, or

\item $n-s+1$ is odd and $d_1\leq s-1 \Rightarrow d_{n-s+1}\geq s+1$,

\end{enumerate}

then $\pi$ is forcibly $3$-component order $s$-connected.\end{cor}

\begin{proof}
When $n-s+1$ is even, the partition is $n-s+1=\left(\frac{n-s+1}{2}\right)\cdot 2$. When $n-s+1$ is odd, the partition is $n-s+1=1+\left(\frac{n-s}{2}\right)\cdot 2$.
\end{proof}

\begin{cor} Let $s$ be an integer and let $\pi=(d_{1}\leq \dots \leq d_{n})$ for $n\geq 4$ and $1\leq s\leq n-2$.
Define $i$ so that $n-s+1\equiv i\bmod{3}$, with $0\leq i\leq 2$. If

\begin{enumerate}

\item $i=0$ and $d_{n-s+1}\geq s+2$, or

\item $i=1$ and $(d_1\leq s-1 \lor d_4\leq s) \Rightarrow d_{n-s+1}\geq s+2$, or

\item $i=2$ and $d_2\leq s \Rightarrow d_{n-s+1}\geq s+2$,

\end{enumerate}

then $\pi$ is forcibly $4$-component order $s$-connected.\end{cor}

\begin{proof}
When $i=0$, the partition of $n-s+1$ is $\left(\frac{n-s+1}{3}\right)\cdot 3$. When $i=1$, the partitions are $1+\left(\frac{n-s}{3}\right)\cdot 3$ and $2\cdot 2+\left(\frac{n-s-3}{3}\right)\cdot 3$. When $i=2$, the partition is $2+\left(\frac{n-s-1}{3}\right)\cdot 3$.
\end{proof}

\begin{cor}\label{keq5} Let $s$ be an integer and let $\pi=(d_{1}\leq \dots \leq d_{n})$ for $n\geq 5$ and $1\leq s\leq n-3$.
Define $i$ so that $n-s+1\equiv i\bmod{4}$, with $0\leq i\leq 3$. If

\begin{enumerate}

\item $i=0$ and $d_{n-s+1}\geq s+3$, or

\item $i=1$ and $(d_1\leq s-1 \lor (d_2\leq s \land d_5\leq s+1) \lor d_9\leq s+1) \Rightarrow d_{n-s+1}\geq s+3$, or

\item $i=2$ and $(d_2\leq s \lor d_6\leq s+1 ) \Rightarrow d_{n-s+1}\geq s+3$, or

\item $i=3$ and $d_3\leq s+1 \Rightarrow d_{n-s+1}\geq s+3$,

\end{enumerate}

then $\pi$ is forcibly $5$-component order $s$-connected.\end{cor}

\begin{proof}
When $i=0$, the partition of $n-s+1$ is $\left(\frac{n-s+1}{4}\right)\cdot 4$. When $i=1$, the partitions are $1+\left(\frac{n-s}{4}\right)\cdot 4$, $2+3+\left(\frac{n-s-4}{4}\right)\cdot 4$, and $3\cdot 3+\left(\frac{n-s-5}{4}\right)\cdot 4$. When $i=2$, the partitions are $2+\left(\frac{n-s-1}{4}\right)\cdot 4$ and $2\cdot 3+\left(\frac{n-s-5}{4}\right)\cdot 4$. When $i=3$, the partition is $3+\left(\frac{n-s-2}{4}\right)\cdot 4$.
\end{proof}

As we can see, the number of conditions grows with $k$ (as a result of there being more sinks). It is shown in \cite{Y14} that the inherent complexity of $k$-component order $s$-connected is at most $F_{k+2}-k$, where $F_i$ is the $i^{\mathrm{th}}$ Fibonacci number when $F_1=F_2=1$.

\begin{obs} Consider $k=n-s+1$ in Theorem~\ref{bmgeqs}. Note that
$$\ds{\sum_{j=1}^{\ell}} c_j = \ceil{n-s+1}{k-1}=\ceil{n-s+1}{n-s}=2.$$
Since $c_j\geq 1$, this implies that either $c_1=c_2=1$ or $c_1=2$. Thus, $n-s+1=h_1+h_2$ for $h_2\geq h_1\geq 1$. We also have $h_2+1\leq h_2+h_1=n-s+1$, which implies $h_2\leq n-s=k-1$. Let $i=h_1$, so that $h_2=n-i-s+1$ and $i\leq n-i-s+1$, or $i\leq \frac{n-s+1}{2}$. Here, the sinks are $G_i=K_{s-1}+(K_i\cup K_{n-i-s+1})$, so that the degree conditions are $d_i\leq i+s-2\,\Rightarrow\, d_{n-s+1}\geq n-i$ for $i\leq \frac{n-s+1}{2}$. These are exactly the degree conditions for the best monotone theorem for $s$-connected, which is consistent with the fact that $G$ is $s$-connected if and only if $\kappa_c^{(n-s+1)}(G)\geq s$~\textnormal{\cite{BGS98}}.
\end{obs}

Note that the sinks for $k$-component order 1-edge connected are the same as for $k$-component order 1-connected. Thus, by setting $s=1$ in Theorem~\ref{bmgeqs}, we get the following theorem.

\begin{thm}\label{bmgeq1} Let $k$ and $n$ be integers with and $n\geq k\geq 2$.
Let $n=c_1 h_1+c_2 h_2 +\cdots+c_{\ell} h_{\ell}$ be a partition of $n$ such that $\ds{\sum_{i=1}^{\ell}} c_i = \ceil{n}{k-1}$ and
$1\leq h_1<h_2<\ldots<h_{\ell}\leq k-1$. Define $m_j=\ds{\sum_{i=1}^{j} c_i h_i}$. The best monotone theorem for
$k$-component order $1$-edge connected on $n$ vertices consists of all conditions of the form
$$\left[\bigwedge_{j=1}^{\ell-1}(d_{m_j}\leq h_j-1)\right]\;\Rightarrow\;d_n\geq h_{\ell}$$ for every such partition of $n$.
\end{thm}

Next, we will determine the sinks for $k$-component order 2-edge connected, and we discuss the inherent complexity for $k$-component order $s$-edge connected when $s\geq 3$.

\begin{obs} Consider two partitions of $n$ having the same length $n=S=h_1+h_2+\cdots +h_{\omega}$ and $n=S'=h'_1+h'_2+\cdots +h'_{\omega}$ such that $h_1<h'_1$, $h_1\leq h_2\leq \cdots\leq h_{\omega}\leq k-1$, and $2\leq h'_1\leq h'_2\leq \cdots\leq h'_{\omega}\leq k-1$. Then $S$ can be achieved from $S'$ by a sequence of transformations in which 1 is subtracted from a part, and 1 is added to another part of equal or greater value.
\end{obs}

\textbf{Example:} Consider $n=21$ and the two partitions $S=2+4+5+5+5$ and $S'=4+4+4+4+5$. A sequence of transformations to get from $S'$ to $S$ is as follows:
$$(4-1)+4+4+(4+1)+5=3+4+4+5+5\longrightarrow (3-1)+4+(4+1)+5+5=2+4+5+5+5=S.$$\exend

\begin{cor}\label{algor} Consider two graphs on $n$ vertices $G=K_{h_1}\cup K_{h_2}\cup\cdots\cup K_{h_{\omega}}$ and $G'=K_{h'_1}\cup K_{h'_2}\cup\cdots\cup K_{h'_{\omega}}$ such that $h_1<h'_1$, $h_1\leq h_2\leq\cdots\leq h_{\omega}$, and $2\leq h'_1\leq h'_2\leq\cdots\leq h'_{\omega}$. Then $G$ can be achieved from $G'$ by a sequence of transformations in which a vertex is disconnected from all vertices in one component and completely joined to another component having equal or greater order.\end{cor}

\textbf{Example:} Consider the two graphs on $n=21$ vertices $G=K_2\cup K_4\cup K_5\cup K_5\cup K_5$ and $G'=K_4\cup K_4\cup K_4\cup K_4\cup K_5$. A sequence of graphs to get from $G'$ to $G$ is as follows:
$$G'=K_4\cup K_4\cup K_4\cup K_4\cup K_5\longrightarrow K_3\cup K_4\cup K_4\cup K_5\cup K_5\longrightarrow K_2\cup K_4\cup K_5\cup K_5\cup K_5=G.$$ \exend

\begin{lem}\label{dsum} Let $G=K_{h_1}\cup K_{h_2}\cup\cdots\cup K_{h_{\omega}}$ and $G'=K_{h'_1}\cup K_{h'_2}\cup\cdots\cup K_{h'_{\omega}}$ with $h_1\leq h_2\leq\cdots\leq h_{\omega}$ and $h'_1\leq h'_2\leq\cdots\leq h'_{\omega}$. Assume $G$ is formed by disconnecting a single vertex from all other vertices in one component of $G'$ and completely joining it to another component of $G'$ having equal or greater order. Then, the degree sum of $G$ is greater than the degree sum of $G'$.\end{lem}

\begin{proof}
Let $i<j$ be integers such that the vertex moved in $G'$ is disconnected from $K_{h'_i}$ and completely joined to $K_{h'_j}$. Then $h_i=h'_i-1$, $h_j=h'_j+1$, and $h_{\ell}=h'_{\ell}$ for $\ell\neq i,j$. The degree sum of $G'$ is
$$\begin{array}{rcl} \ds{\sum_{\ell=1}^{\omega} h'_{\ell}(h'_{\ell}-1)}&=&\ds{\sum_{\ell=1}^{\omega} h_{\ell}(h_{\ell}-1)-h_i(h_i-1)-h_j(h_j-1)+(h_i+1)h_i+(h_j-1)(h_j-2)}\medskip\\ &=&
\ds{\sum_{\ell=1}^{\omega} h_{\ell}(h_{\ell}-1)-2(h_j-h_i-1)}=\ds{\sum_{\ell=1}^{\omega} h_{\ell}(h_{\ell}-1)-2(h'_j-h'_i+1)}\medskip\\ &\leq& \ds{\sum_{\ell=1}^{\omega} h_{\ell}(h_{\ell}-1)-2}<\ds{\sum_{\ell=1}^{\omega} h_{\ell}(h_{\ell}-1)}, \end{array}$$
where the rightmost side of the inequality is the degree sum of $G$.

\end{proof}

\begin{obs} The edge-maximal graphs $G$ with respect to $\lambda_c^{(k)}(G)<2$ are formed by adding a single edge between two components of the edge maximal graphs $H$ for $\lambda_c^{(k)}(H)=0$, i.e., $H=K_{h_1}\cup K_{h_2}\cup \cdots\cup K_{h_{\omega}}$ such that $\omega=\ceil{n}{k-1}$, $n=h_1+h_2+\cdots+h_{\omega}$, and $h_1\leq h_2\leq\cdots\leq h_{\omega}\leq k-1$.\end{obs}

\begin{thm}\label{coecsink} Let $k\geq 2$ and let $G$ be a graph formed by adding a single edge between two components of an edge maximal graph $H$ with respect to $\lambda_c^{(k)}(H)=0$. Then, $\pi(G)$ is a sink for $k$-component order $2$-edge connected, except when $k\geq 4$ and the following are all true: $n\equiv 1\bmod{(k-1)}$, the smallest component of $H$ is $K_2$, and the edge added to $H$ has at most one endpoint in a copy of $K_{k-1}$.
 \end{thm}

\begin{proof}
Consider two graphs $H=K_{h_1}\cup K_{h_2}\cup \cdots\cup K_{h_{\omega}}$ and $H'=K_{h'_1}\cup K_{h'_2}\cup \cdots\cup K_{h'_{\omega}}$ such that $\omega=\ceil{n}{k-1}$, $n=h_1+h_2+\cdots+h_{\omega}=h'_1+h'_2+\cdots+h'_{\omega}$, $h_1\leq h_2\leq\cdots\leq h_{\omega}\leq k-1$, and $h'_1\leq h'_2\leq\cdots\leq h'_{\omega}\leq k-1$. Define $G:=H+e$ and $G':=H'+f$ where $e,f$ are single edges. Also, let $\pi:=\pi(G)=(d_1,d_2,\ldots,d_{n})$ and $\pi':=\pi(G')=(d'_1,d'_2,\ldots,d'_{n})$. If $k=2$, then $H=H'=nK_1$ and $G=G'=(n-2)K_1\cup K_2$. So, there is only one maximal sequence $\pi(G)=0^{n-2}1^2$, and therefore it is a sink. For the remainder of the proof, we will assume $k\geq 3$.

\medskip

\noindent\textit{Case 1}. $H=H'$ and $G\neq G'$.

\medskip

Note that adding a single edge to a graph increases the degree sum by exactly 2. Thus, the degree sum of $G$ and the degree sum of $G'$ are the same. So, neither $\pi$ nor $\pi'$ can majorize the other.

\medskip

\noindent\textit{Case 2}. $H\neq H'$.

\medskip

\noindent\textit{Case 2.1}. $2\leq h_1< h'_1$.

\medskip

By Corollary~\ref{algor} and Lemma~\ref{dsum}, the degree sum of $G$ is greater than the degree sum of $G'$. Thus, $\pi'$ can not majorize $\pi$. Also, $d_1=h_1-1<h'_1-1=d'_1$, so $\pi$ can not majorize $\pi'$.

\medskip

\noindent \textit{Case 2.2}. $2\leq h_1=h'_1$.

\medskip

Let $j\geq 1$ be the first index such that $h_{j+1}\neq h'_{j+1}$, so that $h_i=h'_i$ for $i=1,\ldots,j$. Assume, without loss of generality, that $h_{j+1}<h'_{j+1}$. Define $H_1:=H-\sum_{i=1}^{j} K_{h_i}$ and $H_2:=H'-\sum_{i=1}^{j} K_{h'_i}$, i.e., the graphs that result from removing the first $j$ matching components of $H$ and $H'$, respectively. Since $h_{j+1}<h'_{j+1}$, by Corollary~\ref{algor} and Lemma~\ref{dsum}, the degree sum of $H_2$ is less than the degree sum of $H_1$. Thus, the degree sum of $G'$ is less than the degree sum of $G$, and $\pi'$ can not majorize $\pi$. Also, if $\ell=h_1+h_2+\cdots+h_j$, then $d_{\ell+1}=h_{j+1}-1<h'_{j+1}-1=d'_{\ell+1}$, so $\pi$ can not majorize $\pi'$.

\medskip

\noindent\textit{Case 2.3}. $1=h_1\leq h'_1$.

\medskip

Since $h_1+h_2\geq k$, the only edge-maximal graph $H$ with $h_1=1$ and $\lambda_c^{(k)}(H)=0$ is $H=K_1\cup K_{k-1}\cup\cdots\cup K_{k-1}$, which occurs when $n\equiv 1\bmod(k-1)$. If $h_1=h'_1=1$, then $G$ and $G'$ have the same degree sum and neither $\pi$ nor $\pi'$ majorize the other.

Assume $h_1<h'_1$, then, by Lemma~\ref{dsum}, $\pi'$ can not majorize $\pi$. If $h'_1\geq 3$, then $d_1\leq 1<2\leq h'_1-1=d'_1$, and $\pi$ can not majorize $\pi'$. Next, assume $h'_1=2$. Let the edge $f$ of $G'$ have endpoints in two copies of $K_{k-1}$. If the edge $e$ of $G$ has endpoints in $K_1$ and $K_{k-1}$, then $d_{n-1}=k-2<k-1=d'_{n-1}$, and $\pi$ can not majorize $\pi'$. If instead, $e$ has endpoints in two copies of $K_{k-1}$, then $d_1=0<1=d'_1$, and $\pi$ can not majorize $\pi'$.

Finally, assume $k\geq 4$, $h'_1=2$, and the edge $f$ of $G'$ has endpoints such that at most one is in a copy of $K_{k-1}$. Let the edge $e$ of $G$ have endpoints in $K_1$ and a copy of $K_{k-1}$. Note that $h'_2\neq k-1$, otherwise $n\equiv 2\bmod(k-1)$ instead of $n\equiv 1\bmod(k-1)$. Thus, $h'_2<k-1$. It follows that $d'_1=1=d_1$, $d'_2=h'_2-1< k-2=d_2$, $d'_i\leq k-2=d_i$ for $3\leq i\leq n-1$, and $d'_n\leq k-1=d_n$. So, $\pi$ majorizes $\pi'$, which implies $\pi'$ is not a sink.

\end{proof}

If we generate the sinks according to Theorem~\ref{coecsink} when $k=3,4,$ or 5, we get the following results.

\begin{cor} Let $\pi=(d_{1}\leq \dots \leq d_{n})$ for $n\geq 3$. If

\begin{enumerate}

\item $n$ is even and $d_{n-2}\leq 1\,\Rightarrow\,d_n\geq 3$, or

\item $n$ is odd and $\left[(d_{n-1}\leq 1) \lor (d_1= 0\land d_{n-2}\leq 1)\right]\, \Rightarrow\, d_n\geq 3$,

\end{enumerate}

then $\pi$ is forcibly $3$-component order $2$-edge connected.\end{cor}

\begin{cor} Let $\pi=(d_{1}\leq \dots \leq d_{n})$ for $n\geq 4$.
Define $i$ so that $n\equiv i\bmod{3}$, with $0\leq i\leq 2$. If

\begin{enumerate}

\item $i=0$ and $d_{n-2}\leq 2 \,\Rightarrow\,d_n\geq 4$, or

\item $i=1$ and

$\begin{array}{l}\left[(d_1= 0 \land d_{n-2}\leq 2)\lor (d_1\leq 1\land d_{n-1}\leq 2)\lor
(d_4\leq 1 \land d_{n-2}\leq 2)\right]\Rightarrow d_n\geq 4,\,or\end{array}$

\item $i=2$ and $\left[(d_1\leq 1\land d_{n-1}\leq 2)\lor (d_2\leq 1 \land d_{n-2}\leq 2)\right] \Rightarrow d_n\geq 4$,

\end{enumerate}

then $\pi$ is forcibly $4$-component order $2$-edge connected.\end{cor}

\begin{cor} Let $\pi=(d_{1}\leq \dots \leq d_{n})$ for $n\geq 9$.
Define $i$ so that $n\equiv i\bmod{4}$, with $0\leq i\leq 3$. If

\begin{enumerate}

\item $i=0$ and $d_{n-2}\leq 3 \Rightarrow d_n\geq 5$, or

\item $i=1$ and

$\begin{array}{l}\left[(d_1= 0 \land d_{n-2}\leq 3)\lor (d_1\leq 1\land d_{n-1}\leq 3)\lor
(d_2\leq 1 \land d_5\leq 2 \land d_{n-2}\leq 3) \right. \medskip \\ \left.\lor (d_8\leq 2\land d_{n-1}\leq 3) \lor (d_9\leq 2\land d_{n-2}\leq 3)\right]\Rightarrow d_n\geq 5\end{array}$

and $d_7\leq 2 \Rightarrow d_n\geq 4$, or

\item $i=2$ and

$\begin{array}{l}\left[(d_1\leq 1 \land d_2\leq 2 \land d_{n-1}\leq 3)\lor (d_2\leq 1\land d_{n-2}\leq 3)\lor
(d_5\leq 2 \land d_{n-1}\leq 3) \right. \medskip \\ \left.\lor (d_6\leq 2\land d_{n-2}\leq 3)\right]\Rightarrow d_n\geq 5\end{array}$

and $d_4\leq 2 \Rightarrow d_n\geq 4$, or

\item $i=3$ and $\left[(d_2\leq 2 \land d_{n-1}\leq 3)\lor (d_3\leq 2 \land d_{n-2}\leq 3)\right] \Rightarrow d_n\geq 5$,

\end{enumerate}

then $\pi$ is forcibly $5$-component order $2$-edge connected.\end{cor}

There is a notable increase in the number of sinks from $k$-component order $1$-edge connected to $k$-component order $2$-edge connected. In fact, we have the following result.

\begin{thm}\label{sinkno}
  Let $k,s\geq 3$, and let $n\geq 2(s-1)$. Then there are at least $p(s-1)$ sinks for the property of $k$-component order $s$-edge connected, where~$p$ denotes
  the integer partition function, so that
  $p(z)\sim\dfrac{1}{4\sqrt{3}z}e^{\pi\sqrt{\frac{2z}{3}}}$.
\end{thm}

\begin{proof}
Let $m$ and $r$ be integers such that $m\geq (s-1)/(k-1)$ and $0\leq r\leq k-2$. Consider the graph on $n=2m(k-1)+r$ vertices given by $H=K_r\cup 2mK_{k-1}$. Note that when $n\equiv r\bmod(k-1)$, the graph $H$ has the smallest minimum degree of all edge-maximal graphs with respect to $\lambda_c^{(k)}(H)=0$. Thus, by Lemma~\ref{dsum}, $H$ has the largest degree sum of all such graphs. Let $X$ and $Y$ be two distinct copies of $mK_{k-1}$ within $H$, so that $H=K_r\cup X\cup Y$. Note that $|X|=|Y|\geq s-1$.

Construct an edge-maximal graph $G$ with respect to $\lambda_c^{(k)}(G)=s-1$ by adding $s-1$ edges to $H$ in the following way. Let $a_1+a_2+\cdots +a_j$ be any partition of $s-1$. Choose vertices $x_i$ from $X$, for $1\leq i\leq j$. For each $i$, add $a_i$ edges from $x_i$ to $a_i$ distinct vertices in $Y$ in such a way that no two vertices $x_i$ share a neighbor in $Y$.

Let $G'\neq G$ be another edge-maximal graph with respect to $\lambda_c^{(k)}(G')=s-1$. Then, $G'$ is also formed by starting with an edge-maximal graph $H'$ with respect to $\lambda_c^{(k)}(H')=0$ and adding $s-1$ edges. If $H'\neq H$, then $G$ has a larger degree sum than $G'$, and $\pi(G')$ can not majorize $\pi(G)$. However, if $H'\neq H$, then $G$ and $G'$ have the same degree sum, and again $\pi(G')$ can not majorize $\pi(G)$. Thus, every graph of the form $G$ generates a unique sink $\pi(G)$, of which there are $p(s-1)$ many.

\end{proof}

Theorem~\ref{sinkno} implies that the inherent complexity of the best monotone theorem for $k$-component order $s$-edge connected grows superpolynomially with respect to $s$. Thus, even though we could generate best monotone theorems for $s\geq 3$, the number of conditions may be too unwieldy for practical use.

\section*{Acknowledgments}

The author would like to thank undergraduate research student Kenneth Huang for writing a computer program that provided validation for Theorem~\ref{coecsink}. The program can be found on GitHub at {\color{blue}\mbox{https://github.com/Foldenstein/Theorem-3.13}}.

\end{document}